\documentclass[12pt]{article}

\usepackage{amssymb}
\usepackage{amsthm}
\usepackage{amsmath}
\usepackage{stmaryrd}
\usepackage{titletoc}
\usepackage{mathrsfs}

\newlength{\defbaselineskip}
\newcommand{\setlinespacing}[1]%
           {\setlength{\baselineskip}{#1 \defbaselineskip}}

\theoremstyle{plain}
\newtheorem{thm}{Theorem}[section]

\newtheorem{lem}[thm]{Lemma}

\theoremstyle{definition}

\makeatletter\@addtoreset{equation}{section} \makeatother
\begin{document}

\title {Multi-Dimensional Backward Stochastic
 Differential Equations of Diagonally Quadratic generators\footnotemark[1]}

\author{Ying Hu\footnotemark[2] ~~~~ and ~~~~
Shanjian Tang\footnotemark[3]}
\date{}

\footnotetext[1]{This research was supported by  the National Natural Science Foundation of China (Grants \#10325101 and \#11171076), and by Science and Technology Commission, Shanghai Municipality (Grant No. 14XD1400400).}

\footnotetext[2]{IRMAR, Universit\'e Rennes 1, Campus de Beaulieu,
35042 Rennes Cedex, France.
\textit{E-mail}: \texttt{Ying.Hu@univ-rennes1.fr} (Ying Hu).}

\footnotetext[3]{Department of Finance and Control Sciences, School of
Mathematical Sciences, Fudan University, Shanghai 200433, China.
\texttt{sjtang@fudan.edu.cn} (Shanjian Tang).}

\maketitle

% -----------------------------------------------------------------------

\begin{abstract}
The paper is concerned with  adapted solution of a multi-dimensional BSDE with a ``diagonally" quadratic generator,  the quadratic part of whose $i$th component only depends on the $i$th row of the second unknown variable. Local and global solutions are given.  In our proofs, it is natural and crucial to apply both John-Nirenberg and reverse H\"older inequalities for BMO martingales.
\end{abstract}

AMS Subject Classification: 60H15; 35R60; 93E20

Keywords: Backward stochastic differential equation, Quadratic BSDE, BMO martingale, John-Nirenberg inequality, the reverse H\"older inequality

\section{Introduction}

Consider the following backward stochastic differential equation (BSDE):
\begin{equation}\label{bsde0}
    Y_t=\xi+\int_t^Tg(s,Y_s,Z_s)\, ds-\int_t^T Z_s\, dW_s, \quad t\in [0,T].
\end{equation}
Here, $\{W_t:=(W_t^1, \ldots, W_t^d)^*, 0\le t\le T\}$ is a
$d$-dimensional standard Brownian motion defined on some
probability space $(\Omega, \mathscr{F}, P)$. Denote by $\{\mathscr{F}_t,0\le t\le T\}$ the augmented natural filtration of the
standard Brownian motion $W$. The function $g: \Omega \times [0,T]\times R^n\times R^{n\times d}\to R^n$ is called the generator of BSDE~\eqref{bsde0}. BSDE~\eqref{bsde0} was invented by Bismut in~\cite{Bismut1973} for the linear case,  and in~\cite{Bismut1976} for a specifically structured matrix-valued nonlinear case where the matrix-valued generator contains a quadratic form of the second unknown.  The uniformly Lipschitz case was later studied by Pardoux and Peng~\cite{PardouxPeng1990}.

  Bismut~\cite{Bismut1976} derived a matrix-valued BSDE of a {\it quadratic} generator---the so-called backward stochastic Riccati equation (BSRE) in the study of linear quadratic optimal control with random coefficients, while he could not solve it in general. In that paper, he described the difficulty and failure of his fixed point techniques in the proof of the existence and uniqueness  for BSDE of a quadratic generator (i.e., the so-called quadratic BSDE). It has inspired  subsequent intensive efforts in the research of quadratic BSDE~\eqref{bsde0}. Nowadays numerous progress has been made in this issue: Kobylonski~\cite{Koby2000} and Briand and Hu~\cite{BriandHu2006} gave the existence and uniqueness result for the case of a scalar-valued ($n=1$) quadratic BSDE, Tang~\cite{Tang2003,Tang2014} solved (using the stochastic maximum principle in~\cite{Tang2003} and dynamic programming in~\cite{Tang2014}) the existence and uniqueness result (posed by Bismut~\cite{Bismut1976}) for the general BSRE, and Tevzadze~\cite{Tevzadze} proved the existence and uniqueness result for multi-dimensional quadratic BSDE~\eqref{bsde0} under the assumption that the terminal value is {\it sufficiently small} in the supremum norm (also called the {\it small terminal value problem}). Frei and dos Reis~\cite{FreiReis} constructed a counterexample to show that multi-dimensional quadratic BSDE~\eqref{bsde0} might fail to have a global solution $(Y,Z)$ on $[0,T]$ such that $Y$ is essentially bounded, which illustrates the difficulty of the quadratic part contributing to the underlying scalar generator as an unbounded process---the exponential of whose time-integral is likely to have no finite expectation. Very recently, Cheridito and Nam~\cite{CheriditoNam} addressed a special system of quadratic BSDEs in the Markovian context, which arises from the equilibrium problem with interacting agents in a financial market (see Frei and dos Reis~\cite{FreiReis} for further descriptions). Neither global nor local (in time) positive results are found in the literature for solvability of multidimensional quadratic BSDEs.

Throughout the paper, for $i=1,\ldots, n$, we denote by $z^i$ the $i$th row (component) of matrix (vector) $z$. In the paper, we study the multi-dimensional BSDE~\eqref{bsde0} of the following structured quadratic generator $g:=(g^1,\cdots,g^n)^*$ with
\begin{equation}\label{diagQuadratic}
\begin{split}
g^i(t,y,z)= f^i(t,z^i)+h^i(t,y,z), \quad i=1,\ldots, n,
\end{split}
\end{equation}
for any $(t,y,z)\in [0,T]\times R^n\times R^{n\times d}$, where
\begin{equation}
\begin{split}
|f^i(t, z^i)|\le&\  C(1+|z^i|^2),\\
|h^i(t,y,z)|\le&\  C(1+|y|+|z|^{1+\alpha}), \quad \alpha\in [0,1), \quad i=1,\ldots,n;
\end{split}
\end{equation}
or equivalently where
\begin{equation}
\begin{split}
|g^i(t,y,z)|\le&\  C(1+|y|+|z|^{1+\alpha}+|z^i|^2), \quad \alpha\in [0,1), \quad i=1,\ldots,n.
\end{split}
\end{equation}
This kind of structured  generator $g$ is said to be ``diagonally" quadratic. Assuming some additional (locally or globally) Lipschitz continuity in both unknowns of the generator $g$,
we prove that for $\alpha\in [0,1)$ and a given bounded terminal value  $\xi$, multi-dimensional diagonally quadratic BSDE~\eqref{bsde0} admits a unique local solution $(Y,Z)$ on $[T-\varepsilon, T]$ ($\varepsilon>0$)  such that $Y$ is bounded and $Z\cdot M$ is a BMO martingale (see Theorem~\ref{thm 1}), and moreover, it has a unique global solution on $[0,T]$ in the case of $\alpha=0$ (see Theorem~\ref{thm 2}).

The rest of the paper is organized as follows. In Section 1, we introduce the background for quadratic BSDEs and some related (rather than exhausted) studies. In Section 2, we prepare some notations and known inequalities for BMO martingales. We also prove existence, uniqueness, comparison, and a priori estimate for one-dimensional quadratic BSDEs of possibly unbounded generators, and state the main results of the paper. In Section 3, we prove our local solution for our multi-dimensional diagonally quadratic BSDE~\eqref{bsde0}. Finally, in Section 4, we prove our global solution to our multi-dimensional diagonally quadratic BSDE~\eqref{bsde0}.

\section{Preliminaries and statement of main results}

Let $M=(M_t, \mathscr{F}_t)$ be a uniformly integrable martingale with $M_0=0$, and for
$p\in [1, \infty)$ we set
\begin{equation}
\begin{split}
 \|M\|_{BMO_p(P)}:= \left|\sup_{\tau} E^{\mathscr{F}_\tau}\left[\left(\langle M\rangle_\tau^\infty\right)^{p\over2}\right]^{1\over p}\right|_\infty
\end{split}
\end{equation}
where the supremum is taken over all stopping times $\tau$. The class $\{M : \|M\|_{BMO_p} <
\infty\}$ is denoted by $BMO_p$, which is  written into  $BMO_p(P)$ whenever it is necessary to indicate the underlying probability, and observe that $\|\cdot\|_{BMO_p}$ is a norm on this space and $BMO_2(P)$ is a Banach space. 

Denote by $\mathscr{E}(M)$ the stochastic exponential of a local martingale $M$ and by $\mathscr{E}(M)_s^t$ that of $M_t-M_s$. Denote by $\beta \cdot M$ the stochastic integral of an adapted process $\beta$ with respect to the local continuous martingale $M$.  Define
$$\log^+(x):=0\vee \log x, \quad x\in (0, \infty).$$ We have
$$\log^+ x\le \log (1+x), \quad x\in (0, \infty).$$

Denote by $\mathscr{S}^\infty(R^n)$ the totality of $R^n$-valued $\mathscr{F}_t$-adapted essentially bounded continuous processes, and by $|Y|_\infty$ the essential supremum norm of $Y\in \mathscr{S}^\infty(R^n)$. It can be verified that $(\mathscr{S}^\infty(R^n), |\cdot|_\infty)$ is a Banach space.

\begin{lem} (John-Nirenberge inequality) If $\|M\|_{BMO_2}<1$, then for any stopping time $\tau$,
\begin{equation}\label{JohnNir}
    E^{\mathscr{F}_\tau}\left[e^{\langle M\rangle_\tau^\infty}\right]\le {1\over 1-\|M\|^2_{BMO_2}}.
\end{equation}
\end{lem}

As a corollary of the last lemma, we have
\begin{lem}  For any $p\in [1, \infty)$, there is a generic constant $L_p>0$ such that for any uniformly integrable martingale $M$,
\begin{equation}\label{norm equivalence}
    \|M\|_{BMO_p}\le L_p  \|M\|_{BMO_2}.
\end{equation}
\end{lem}

Define
$$
\Phi(x):= \left\{1+{1\over  x^2}\log{2x-1\over 2(x-1)}\right\}^{1\over2}-1, \quad x>1.
$$
It is clearly continuous and decreasing, satisfying $\Phi(1+0) =+\infty$ and $\Phi(\infty)=0$.

\begin{lem} (The reverse H\"older inequality) Let $p\in (1, \infty)$.  If $\|M\|_{BMO_2}<\Phi(p)$, then $\mathscr{E}(M)$ satisfies the reverse H\"older inequality $(R_p)$:
\begin{equation}\label{Rp}
    E^{\mathscr{F}_\tau}\left[\mathscr{E}(M)_\tau^\infty\right]^p\le c_p
\end{equation}
for  any stopping time $\tau$, with a constant $c_p>0$ depending only on $p$.
\end{lem}

All the proofs of the preceding three lemmas can be found in Kazamakie~\cite{Kazamaki}. The following lemma plays an important role in our subsequent arguments. It indicates that  following the proof of \cite[Theorem 3.3, page 57]{Kazamaki} can give a preciser dependence of the two constants $c_1, c_2$ on $\beta\cdot M$. For reader's convenience, we also give a complete proof here.

\begin{lem} For $K>0$, there are constants $c_1>0$ and $c_2>0$ such that for any BMO martingale $M$, we have for any BMO martingale $N$ such that $\|N\|_{BMO_2(P)}\le K$,
\begin{equation}\label{uniform BMO estimate}
c_1\|M\|_{BMO_2(P)}\le \|\widetilde M\|_{BMO_2(\widetilde P)}\le c_2\|M\|_{BMO_2(P)}
\end{equation}
where ${\widetilde M}:=M-\langle M, N\rangle$ and $d\widetilde P:= \mathscr{E}(N)_0^\infty d P$.
\end{lem}

\begin{proof}
(i) {\it Derivation of the second inequality.}  From Chapter 3 of Kazamaki, there is $p>1$ such that  $\Phi(p)>K$, which gives in view of Lemma~\ref{Rp} that for any stopping time $\tau$,
$$
E^{\mathscr{F}_\tau}\left[\left(\mathscr{E}(N)_\tau^\infty\right)^p\right]\le c_p.
$$
Therefore,
\begin{eqnarray}
    \begin{split}
\|\widetilde M\|^2_{BMO_2(\widetilde P)}=&\ \sup_{\tau}E^{\mathscr{F}_\tau}\left[\langle M\rangle^\infty_\tau\mathscr{E}(N)_\tau^\infty\right]\\
\le &\ \sup_{\tau} E^{\mathscr{F}_\tau} \left[\left(\langle M\rangle_\tau^\infty\right)^q\right]^{1\over q}c_p^{1\over p}\\
\le &\  L_{2q}^2c_p^{1\over p}\|M\|^2_{BMO_2(P)}.
\end{split}
\end{eqnarray}

(ii) {\it Derivation of the first inequality.} From the proof of \cite[Theorem 3.3, page 57]{Kazamaki}, we have
\begin{eqnarray}
    \begin{split}
&\ {\widetilde E}^{\mathscr{F}_\tau}\left[\left(\mathscr{E}(-{\widetilde N})_\infty^\tau \right)^{1\over q-1}\right]\\
=&\ E^{\mathscr{F}_\tau}\left[\left(\mathscr{E}(-{\widetilde N})_\infty^\tau \right)^{1\over q-1}\mathscr{E}(N)_\tau^\infty\right] \\
=&\ E^{\mathscr{F}_\tau}\left[(\mathscr{E}(N)_\tau^\infty)^{1+{1\over q-1}}\right]\\
=&\ E^{\mathscr{F}_\tau}\left[(\mathscr{E}(N)_\tau^\infty)^p\right]\le c_p.
\end{split}
\end{eqnarray}
Here and in the following, we denote by $\widetilde E$ the expectation operator with respect to the probability $\widetilde P$.
From the proof of \cite[Theorem 2.4, pages 33--34]{Kazamaki}, we have
\begin{eqnarray}
    \begin{split}
    &e^{{1\over q-1}{\widetilde E}^{\mathscr{F}_\tau}\left[\log^+\left(\mathscr{E}(-{\widetilde N})_\infty^\tau \right)\right]}\\
\le &\ e^{{\widetilde E}^{\mathscr{F}_\tau}\left[\log^+\left(\mathscr{E}(-{\widetilde N})_\infty^\tau \right)^{1\over q-1}\right]}\\
\le &\ e^{{\widetilde E}^{\mathscr{F}_\tau}\left[\log\left(\left(\mathscr{E}(-{\widetilde N})_\infty^\tau \right)^{1\over q-1}+1\right)\right]}\\
\le & {\widetilde E}^{\mathscr{F}_\tau}\left[\left(\mathscr{E}(-{\widetilde N})_\infty^\tau \right)^{1\over q-1}\right]+1\le c_p+1,
\end{split}
\end{eqnarray}
which gives the following inequality;
\begin{eqnarray}
    \begin{split}
{\widetilde E}^{\mathscr{F}_\tau}\left[\log^+\left(\mathscr{E}(-{\widetilde N})_\infty^\tau \right)\right]\le &\ (q-1)\log (c_p+1).
\end{split}
\end{eqnarray}
Therefore, we have
\begin{eqnarray}
    \begin{split}
& {\widetilde E}^{\mathscr{F}_\tau}\left[\langle N\rangle_\tau^\infty\right]
\le \ 2{\widetilde E}^{\mathscr{F}_\tau}\left[\widetilde N^\infty-\widetilde N_\tau+{1\over2}\langle \widetilde N\rangle_\tau^\infty\right]\\
\le &\ 2{\widetilde E}^{\mathscr{F}_\tau}\left[\log^+\left(\mathscr{E}(-{\widetilde N})_\infty^\tau \right)\right]
\le {\bar K}^2 :=\ 2(q-1)\log (c_p+1), \quad \text{ with $\bar K\ge 0$}.
\end{split}
\end{eqnarray}

Now choose $\bar p>1$ such that $\Phi(\bar p)> \bar K$. Then, we have $\Phi(\bar p)> \|-{\widetilde N}\|_{BMO_2(P^i)}$ and the following reverse H\"older inequality ($R_{\bar p}$) holds:
$$
{\widetilde E}^{\mathscr{F}_\tau}\left[\left(\mathscr{E}(-{\widetilde N})_\tau^\infty\right)^{\bar p}\right]\le c_{\bar p}.
$$
Therefore, we have
\begin{eqnarray}
    \begin{split}
&\ \|M\|^2_{BMO_2(P)}=\sup_{\tau}E^{\mathscr{F}_\tau}\left[\langle M\rangle_\tau^\infty\right]\\
=&\ \sup_{\tau} {\widetilde E}^{\mathscr{F}_\tau}\left[\langle M\rangle_\tau^\infty\mathscr{E}(-{\widetilde N})_\tau^\infty\right]\\
\le &\ \sup_{\tau} {\widetilde E}^{\mathscr{F}_\tau} \left[\left(\langle M\rangle_\tau^\infty\right)^{\bar q}\right]^{1\over \bar q}c_{\bar p}^{1\over \bar p}\\
\le &\ L_{2\bar q}^2c_{\bar p}^{1\over \bar p}\|M\|^2_{BMO_2(\widetilde P)}.
\end{split}
\end{eqnarray}
\end{proof}

Define
\begin{equation}
\phi(y):=\gamma^{-2}[\exp{(\gamma |y|)}-\gamma |y|-1], \quad y\in R.
\end{equation}
Then, we have for $y\in R, $
\begin{equation}
\phi'(y)=\gamma^{-1}[\exp{(\gamma |y|)}-1]\mbox{sgn}(y), \quad \phi''(y)=\exp{(\gamma |y|)}, \quad \phi''(y)-\gamma |\phi'(y)|=1.
\end{equation}

We have the following existence and uniqueness, a priori estimate, and comparison for one-dimensional BSDEs with unbounded data.

\begin{lem}\label{Quad bsde} Assume that (i) the function $f: \Omega\times [0,T]\times R^d\to R$  has the following quadratic growth and locally Lipschitz continuity in the last variable:
\begin{equation}
\begin{split}
|f(s, z)|\le&\  C+{\gamma\over2}|z|^2, \quad z\in R^d;\\
|f(s,z_1)-f(s,z_2)|\le & C(1+|z_1|+|z_2|)|z_1-z_2|, \quad z_1,z_2\in R^d;
\end{split}
\end{equation}
(ii) the process $f(\cdot, z)$ is $\mathscr{F}_t$-adapted for each $z\in R^d$; and (iii) the process $g: \Omega\times [0,T]\to R$ is $\mathscr{F}_t$-adapted and $|g_s|\le |H_s|^{1+\alpha}$ such that the stochastic integral $H\cdot W$ is a BMO martingale. Then for bounded $\xi$, the following BSDE
\begin{equation}\label{new bsde}
Y_t=\xi+\int_t^T [f(s, Z_s)+g_s]\, ds-\int_t^TZ_s\, dW_s, \quad t\in [0,T]
\end{equation}
has a unique solution $(Y, Z)$ such that $Y$ is (essentially) bounded and $Z\cdot W$ is a BMO martingale. Further more, we have
$$
e^{\gamma |Y_t|}\le E^{\mathscr{F}_t}\left[e^{\gamma \xi+\gamma\int_t^T|g(s)|\, ds }\right].
$$
Let $(\widetilde{Y}, \widetilde{Z})$ solve the following BSDE:
$$
Y_t=\widetilde{\xi}+\int_t^T [{\widetilde f}(s, Z_s)+g(s)]\, ds-\int_t^TZ_s\, dW_s, \quad t\in [0,T]
$$
where the pair $(\widetilde f, \widetilde \xi)$ has the same above-mentioned properties to $(f, \xi)$,  $\widetilde{f}\ge f$,  and $\widetilde{\xi}\ge \xi$. Then we have $\widetilde{Y}_t\ge Y_t$.
\end{lem}

\begin{proof} Let $\delta>0$ be sufficiently small such that ${1+\alpha\over 2}\delta \|H\cdot W\|^2_{BMO_2}<1$. From John-Nirenberg inequality~\eqref{JohnNir}, we have
\begin{eqnarray}
    \begin{split}
       E^{\mathscr{F}_t}\left[e^{{1+\alpha\over 2}\int_t^T\delta H_s^2\, ds}\right]
\le \ {1\over 1-{{1+\alpha}\over 2}\delta \|H\cdot W\|^2_{BMO}}.
\end{split}
\end{eqnarray}

Since
\begin{eqnarray}\label{delta}
    \begin{split}
\gamma |g|\le \gamma H^{1+\alpha}=&\ (\sqrt{\delta}H)^{1+\alpha}\cdot \delta^{-{1+\alpha\over2}}
\le\ {1+\alpha\over 2}\delta H^2+\delta_\alpha
\end{split}
\end{eqnarray}
with
$$\delta_\alpha:={1-\alpha\over 2}(\gamma\delta^{-{1+\alpha\over2}})^{2\over 1-\alpha}={1\over 2}\gamma^{2\over 1-\alpha}\delta^{-{1+\alpha\over 1-\alpha}}(1-\alpha),
$$
we have for some $p>1$,
\begin{eqnarray}
    \begin{split}
    &E\left[e^{p\gamma |\xi|+p\gamma \int_0^T|g_s|\, ds}\right]\\
     \le &\ e^{(\gamma |\xi|_\infty+)T}
    E^{\mathscr{F}_t}\left[e^{{1+\alpha\over 2}\int_t^T\delta H_s^2\, ds}\right] \\
\le &\ {1\over 1-{{1+\alpha}\over 2}\delta \|H\cdot W\|^2_{BMO}} e^{(\gamma |\xi|_\infty+\delta_\alpha)T}<\infty.
\end{split}
\end{eqnarray}

From Briand and Hu~\cite{BriandHu2006}, BSDE~\eqref{new bsde} admits a solution $(Y,Z)$, satisfying
\begin{eqnarray}
    \begin{split}
     e^{\gamma |Y_t|}\le &\  E^{\mathscr{F}_t}\left[e^{\gamma |\xi|+\int_t^T\gamma |g_s|\, ds}\right]
     \le \ {1\over 1-{{1+\alpha}\over 2}\delta \|H\cdot W\|^2_{BMO}} e^{(\gamma |\xi|_\infty+\delta_\alpha)T}.
\end{split}
\end{eqnarray}
This shows that $Y$ is bounded.

We now show that $Z\cdot W$ is a BMO martingale.  Using It\^o's formula to compute $\phi(Y_t)$, we have
\begin{eqnarray}
    \begin{split}
      &\phi(Y_t)+{1\over2} E^{\mathscr{F}_t}\left[\int_t^T\phi''(Y_s)|Z_s|^2\, ds\right] \\
      = \, &E^{\mathscr{F}_t}[\phi(\xi)] +E^{\mathscr{F}_t}\left[\int_t^T\phi'(Y_s)\left(f(s, Z_s)+g_s\right)\, ds\right]\\
      \le\, & \phi(|\xi|_\infty)+ E^{\mathscr{F}_t}\left[\int_t^T|\phi'(Y_s)|\left({\gamma\over2}|Z_s|^2+ |g_s|\right)\, ds\right].
    \end{split}
\end{eqnarray}
Therefore,
\begin{eqnarray}
    \begin{split}
      &\phi(Y_t)+{1\over2} E^{\mathscr{F}_t}\left[\int_t^T|Z_s|^2\, ds\right]
      \le\,  \phi(|\xi|_\infty)+ E^{\mathscr{F}_t}\left[\int_t^T|\phi'(Y_s)||g_s|\, ds\right]\\
      \le & \ \phi(|\xi|_\infty)+\phi'(|Y|_\infty) E^{\mathscr{F}_t}\left[\int_t^T|g_s|\, ds\right].
 \end{split}
\end{eqnarray}
In view of the estimate~\eqref{delta} (for $\delta=1$), we have
\begin{eqnarray}
    \begin{split}
     &\phi(Y_t)+{1\over2} E^{\mathscr{F}_t}\left[\int_t^T|Z_s|^2\, ds\right]\\
      \le & \ \phi(|\xi|_\infty)+\phi'(|Y|_\infty) \left[{1+\alpha\over 2\gamma}\  E^{\mathscr{F}_t}\!\!\int_t^T\!\! H^2\, ds+{1-\alpha\over 2\gamma}\gamma^{2\over 1-\alpha}T\right]\\
       \le & \ \phi(|\xi|_\infty)+\phi'(|Y|_\infty) \left[{1+\alpha\over 2\gamma}\|H\cdot W\|^2_{BMO_2(P)}+{1-\alpha\over 2}\gamma^{1+\alpha\over 1-\alpha}T\right],
    \end{split}
\end{eqnarray}
which yields that $Z\cdot W$ is a BMO-martingale.

Now we compare two pairs of solutions. Define
$$
\delta Y:= \widetilde Y-Y, \quad \delta Z:= \widetilde Z-Z.
$$
We have
\begin{eqnarray}
    \begin{split}
\delta Y_t=&\ \widetilde{\xi}-\xi+\int_t^T[\widetilde{f}(s, \widetilde{Z}_s)-\widetilde{f}(s, Z_s)+\widetilde{f}(s, Z_s)-f(s, Z_s)]\, ds-\int_t^T\delta Z_s\, dW_s.
\end{split}
\end{eqnarray}
In a straightforward way,  for some adapted process $\beta$ such that $|\beta_s|\le C(1+|Z_s|)$ (and therefore $\beta\cdot W$ is a BMO martingale), the last equation is written into the following one
\begin{eqnarray}
    \begin{split}
\delta Y_t=&\ \widetilde{\xi}-\xi+\int_t^T[\delta Z_s\beta_s +\widetilde{f}(s, Z_s)-f(s, Z_s)]\, ds-\int_t^T \delta Z_s \, dW_s.
\end{split}
\end{eqnarray}
Define
\begin{eqnarray}
    \begin{split}
\widetilde{W}_t:=W_t-\int_0^t\beta_s\, ds, \quad t\in [0,T]; \quad d\widetilde P:=\mathscr{E}(\beta\cdot W)_0^T dP.
\end{split}
\end{eqnarray}
Then, $\widetilde P$ is a new probability, and $\widetilde W$ is a Brownian motion with respect to $\widetilde P$. We have
\begin{eqnarray}
    \begin{split}
\delta Y_t= &\ \widetilde{\xi}-\xi+\int_t^T[\widetilde{f}(s, Z_s)-f(s, Z_s)]\, ds-\int_t^T\delta Z_s \, d\widetilde{W}_s.
\end{split}
\end{eqnarray}
Taking the obvious conditional expectation with respect to $\widetilde P$, we have the desired inequality $\delta Y_t\ge 0$. The uniqueness result follows immediately from the comparison result.
\end{proof}

We make the following three assumptions.

\medskip
($\mathscr{A}1$) The function $f:=(f^1, \ldots, f^n)^*: \Omega\times [0,T]\times R^d\to R^n$  has the following quadratic growth and locally Lipschitz continuity in the last variable:
\begin{equation}
\begin{split}
|f(s, z)|\le&\  C+{\gamma\over2}|z|^2, \quad z\in R^d;\\
|f(s,z_1)-f(s,z_2)|\le & C(1+|z_1|+|z_2|)|z_1-z_2|, \quad z_1,z_2\in R^d.
\end{split}
\end{equation}
For each $z\in R^d$, the process $f(\cdot, z)$ is $\mathscr{F}_t$-adapted.

($\mathscr{A}2$)  There is $\alpha\in [0,1)$ such that the function $h:=(h^1, \ldots, h^n)^*: \Omega\times [0,T]\times R^n\times R^{n\times d}\to R^n$  has the following quadratic growth and Lipschitz continuity in the last two variables:
\begin{equation}
\begin{split}
 \ &|h(t,y,z)|\le C(1+|y|+|z|^{1+\alpha}), \quad (y,z)\in  R^n\times R^{n\times d}; \\
&|h(t,y_1, z_1)-h(t,y_2,z_2)|\le C|y_1-y_2|+C(1+|z_1|^\alpha+|z_2|^\alpha)|z_1-z_2|,
\end{split}
\end{equation}
for $(y_j, z_j)\in  R^n\times R^{n\times d}$ with $j=1,2$. For each $(y,z)\in R^n\times R^{n\times d}$, the process $g(\cdot, y, z)$ is $\mathscr{F}_t$-adapted.

\medskip
($\mathscr{A}3$) The terminal condition $\xi:=(\xi^1, \ldots, \xi^n)^*$ is uniformly bounded.

\medskip

The main results of the paper are the following two theorems.

\begin{thm}\label{thm 1} Let assumptions $(\mathscr{A}1), (\mathscr{A}2),$ and $(\mathscr{A}3)$ be satisfied with $\alpha\in [0,1)$. Then, for any bounded $\xi$, BSDE~\eqref{bsde0} with generator $g$ satisfying~\eqref{diagQuadratic} has a unique local solution $(Y,Z)$.
\end{thm}

\medskip

\begin{thm}\label{thm 2} Let assumptions $(\mathscr{A}1)$ and $(\mathscr{A}3)$ be satisfied. Moreover, assume that there is a positive constant $C$ such that for $(s,y,z)\in [0,T]\times R^n\times R^{n\times d}$ and $(\bar y, \bar z)\in R^n\times R^{n\times d}$,
\begin{eqnarray}
    \begin{split}
|h(s,y,z)|\le C(1+|y|+|z|), \quad |h(s,y,z)-h(s,\bar y, \bar z)|\le C(|y-\bar y|+|z- \bar z|).
\end{split}
\end{eqnarray}
Then, the following BSDE
\begin{eqnarray}
    \begin{split}
Y_t^i=\, & \xi^i+\int_t^T \left[f^i(z^i_s)+h^i(s,Y_s,Z_s)\right]\, ds-\int_t^T Z_s^i\, dW_s, \quad t\in [0,T]; \quad i=1, \ldots,n
\end{split}
\end{eqnarray}
has a unique adapted solution $(Y, Z)$ on $[0, T]$ such that $Y$ is bounded. Further more,  $Z\cdot W$ is a $BMO(P)$ martingale.
\end{thm}

\section{Local solution: the proof of Theorem~\ref{thm 1}}

For a pair of bounded adapted process $U$ and BMO martingale $V\cdot W$, in view of Lemma~\ref{Quad bsde}, the following decoupled system of quadratic BSDEs:
\begin{equation}\label{qbsde2}
Y_t^i=\xi^i+\int_t^T\left(f^i(s, Z_s^i)+h^i(U_s,V_s)\right)\, ds -\int_t^TZ_s^i\, dW_s, \ t\in [0, T]; i=1,\ldots,n,
\end{equation}
has a unique adapted solution $(Y,Z)$ such that $Y$ is bounded and $Z\cdot W$ is a BMO martingale. Define the quadratic solution map $\Gamma: (U,V)\mapsto \Gamma (U,V)$ as follows:
$$\Gamma(U,V):=(Y,Z), \quad \forall (U, V\cdot W)\in \mathscr{S}^\infty(R^n)\times (BMO_2(P))^n.$$
It is a transformation in the Banach space $\mathscr{S}^\infty(R^n)\times (BMO_2(P))^n$.

Define
\begin{equation}\label{C delta}
     C_\delta:=e^{{6\over 1-\alpha}\gamma CT+ {3\over 2} \gamma C \left({n\over \delta}\right)^{1+\alpha\over 2}T};
\end{equation}
\begin{equation}\label{beta}
    \beta:={1\over 2}(1-\alpha)C^{2\over 1-\alpha}(2(1+\alpha))^{1+\alpha\over 1-\alpha};
\end{equation}
\begin{eqnarray}\label{mu1,2}
    \begin{split}
\label{lambda}
    \mu_1:=&\ (1-\alpha)\left(1+{1-\alpha\over (1+\alpha)\gamma }\right)= 1-\alpha+{(1-\alpha)^2\over  (1+\alpha)\gamma} ,\\
     \mu_2:=&\ (1+\alpha)\left(1+{1-\alpha\over  (1+\alpha)\gamma }\right)= 1+\alpha+{1-\alpha\over \gamma};
\end{split}
\end{eqnarray}
\begin{equation}\label{mu}
\mu:= \left(\beta+C\mu_1\right)\gamma^{2\over \alpha-1}+C\mu_2.
\end{equation}
Consider the following standard quadratic equation of $A$:
$$
\delta A^2-\left(1+4Cn\gamma^{-2}e^{\gamma |\xi|_\infty} \delta\right)A +4Cn\gamma^{-2}e^{\gamma |\xi|_\infty} +4\mu C_\delta e^{{3\over 1-\alpha}n\gamma |\xi|_\infty} \varepsilon=0.
$$
The discriminant of the quadratic equation reads
\begin{eqnarray}
    \begin{split}
\Delta:=&\left(1+4Cn\gamma^{-2}e^{\gamma |\xi|_\infty} \delta\right)^2-4\delta \left[ 4Cn\gamma^{-2}e^{\gamma |\xi|_\infty} +4\mu C_\delta e^{{3\over 1-\alpha}n\gamma |\xi|_\infty} \varepsilon\right]\\
=&\left(1-4Cn\gamma^{-2}e^{\gamma |\xi|_\infty} \delta\right)^2-16\mu \delta C_\delta e^{{3\over 1-\alpha}n\gamma |\xi|_\infty} \varepsilon.
\end{split}
\end{eqnarray}
Take
\begin{eqnarray}\label{A}
    \begin{split}
\delta:= &{1\over 8Cn}\gamma^2 e^{-\gamma |\xi|_\infty}, \quad \varepsilon\le  \min\left\{{1\over 3nC}, {C n\over 8\mu C_\delta}\gamma^{-2} e^{(1-{3n\over 1-\alpha})\gamma |\xi|_\infty}\right\},\\
 A:=& {{1+4Cn\gamma^{-2}e^{\gamma |\xi|_\infty} \delta-\sqrt{\Delta}}\over 2 \delta}={3-2\sqrt{\Delta}\over 4\delta }\le {3\over 4\delta}={3\over 2} Ce^{\gamma |\xi|_\infty},
\end{split}
\end{eqnarray}
and we have
\begin{eqnarray}\label{root}
    \begin{split}
&\Delta \ge 0, \quad 1-\delta A={1+2\sqrt{\Delta}\over 4}>0, \\
&{Cn\over \gamma^2}e^{\gamma |\xi|_\infty}+\mu C_\delta{e^{{3\over 1-\alpha}n\gamma |\xi|_\infty}\over 1-\delta A}\varepsilon +{1\over4}A= {1\over2} A.
\end{split}
\end{eqnarray}

Throughout this section, we base our discussion on the time interval $[T-\varepsilon, T].$ 

We shall prove Theorem~\ref{thm 1} by showing that the quadratic solution map $\Gamma$ is a contraction on the ball $\mathscr{B}_{\varepsilon}$ defined by
\begin{eqnarray}
  \mathscr{B}_{\varepsilon}:=\left\{ (U,V): \|V\cdot W\|^2_{BMO_2}\le A,\ \   e^{{2\over 1-\alpha}n\gamma |U|_\infty}\le {C_\delta e^{{3\over 1-\alpha}n\gamma |\xi|_\infty}\over 1-\delta A}\right\}
\end{eqnarray}
for a positive constant  $\varepsilon$ (to be determined later).

\subsection{Estimation of the quadratic solution map}

We shall show the following assertion: $\Gamma(\mathscr{B}_{\varepsilon})\subset \mathscr{B}_{\varepsilon}$, that is,
\begin{equation}\label{desired}
\Gamma(U,V)\in \mathscr{B}_{\varepsilon}, \quad\quad \forall\ (U,V)\in \mathscr{B}_{\varepsilon}.
\end{equation}

{\it Step 1. Exponential transformation.}

\medskip
For each $i=1,\ldots,n$, using It\^o's formula to compute $\phi(Y_t^i)$, we have
\begin{eqnarray}\label{qbsde3}
    \begin{split}
      &\phi(Y_t^i)+{1\over2} E^{\mathscr{F}_t}\left[\int_t^T\phi''(Y_s^i)|Z_s^i|^2\, ds\right] \\
      = \, &E^{\mathscr{F}_t}[\phi(\xi^i)] +E^{\mathscr{F}_t}\left[\int_t^T\phi'(Y_s^i)\left(f^i(s, z_s^i)+h(U_s,V_s)\right)\, ds\right]\\
      \le\, & \phi(|\xi^i|_\infty)+ E^{\mathscr{F}_t}\left[\int_t^T|\phi'(Y_s^i)|\left(C+{\gamma\over2}|Z_s^i|^2+C(1+|U_s|+|V_s|^{1+\alpha})\right)\, ds\right].
    \end{split}
\end{eqnarray}
Therefore,
\begin{eqnarray}\label{quadratic estimate}
    \begin{split}
      &\phi(Y_t^i)+{1\over2} E^{\mathscr{F}_t}\left[\int_t^T|Z_s^i|^2\, ds\right] \\
      \le\, & \phi(|\xi^i|_\infty)+ CE^{\mathscr{F}_t}\left[\int_t^T|\phi'(Y_s^i)|\left(2+|U_s|+|V_s|^{1+\alpha}\right)\, ds\right].
    \end{split}
\end{eqnarray}
For $t\in [T-\epsilon, T],$ we have
\begin{eqnarray}
    \begin{split}
      &\sum_{i=1}^n\phi(Y_t^i)+{1\over2} E^{\mathscr{F}_t}\left[\int_t^T|Z_s|^2\, ds\right] \\
      \le\, & \sum_{i=1}^n\phi(|\xi^i|_\infty)+C\sum_{i=1}^n E^{\mathscr{F}_t}\left[\int_t^T|\phi'(Y_s^i)|\left(2+|U_s|+|V_s|^{1+\alpha}\right)\, ds\right]
   \end{split}
\end{eqnarray}
Since (in view of the definition of notation $\beta$ in~\eqref{beta})
$$
C\phi'(Y_s^i) |V_s|^{1+\alpha}\le {1\over4}|V_s|^2+\beta |\phi'(Y_s^i)|^{2\over 1-\alpha},
$$
we have
\begin{eqnarray}\label{Young}
    \begin{split}
      &\sum_{i=1}^n\phi(Y_t^i)+{1\over2} E^{\mathscr{F}_t}\left[\int_t^T|Z_s|^2\, ds\right] \\
        \le\, & C\sum_{i=1}^n\phi(|\xi^i|_\infty)+C\sum_{i=1}^n E^{\mathscr{F}_t}\left[\int_t^T|\phi'(Y_s^i)|\left(2+|U_s|\right)\, ds\right]\\
                &+\beta\sum_{i=1}^n E^{\mathscr{F}_t}\left[\int_t^T|\phi'(Y_s^i)|^{{2\over 1-\alpha}}\, ds\right]+{1\over4}E^{\mathscr{F}_t}\left[\int_t^T|V_s|^2\, ds\right]\\
                 \le\, & C\sum_{i=1}^n\phi(|\xi^i|_\infty)+2C\sum_{i=1}^n E^{\mathscr{F}_t}\left[\int_t^T|\phi'(Y_s^i)|\left(1+|U_s|\right)\, ds\right]\\
        &+\beta\sum_{i=1}^n E^{\mathscr{F}_t}\left[\int_t^T|\phi'(Y_s^i)|^{{2\over 1-\alpha}}\, ds\right]+{1\over4}E^{\mathscr{F}_t}\left[\int_t^T|V_s|^2\, ds\right].
   \end{split}
\end{eqnarray}
In view of the inequality for $x>0,$
$$
1+x\le \left(1+{1-\alpha\over \gamma(1+\alpha)}\right)e^{{\gamma(1+\alpha)\over 1-\alpha}x},
$$
we have
\begin{eqnarray}
    \begin{split}
&2C\sum_{i=1}^n E^{\mathscr{F}_t}\left[\int_t^T|\phi'(Y_s^i)|\left(1+|U_s|\right)\, ds\right] \\
\le & 2C\sum_{i=1}^n E^{\mathscr{F}_t}\left[\int_t^T|\phi'(Y_s^i)|\left(1+{1-\alpha\over \gamma (1+\alpha)}\right)e^{\gamma {1+\alpha\over 1-\alpha}|U_s|}\, ds\right].
\end{split}
\end{eqnarray}
Since (by Young's inequality)
$$
|\phi'(Y_s^i)|e^{\gamma {1+\alpha\over 1-\alpha}|U_s|}\le {1-\alpha\over2}|\phi'(Y_s^i)|^{2\over 1-\alpha}+{1+\alpha\over2}e^{{2 \over 1-\alpha}\gamma|U_s|},
$$
in view of the definition of the notations $\mu_1$ and $\mu_2$ in~\eqref{mu1,2}, we have
\begin{eqnarray}
    \begin{split}
&2C\sum_{i=1}^n E^{\mathscr{F}_t}\left[\int_t^T|\phi'(Y_s^i)|\left(1+|U_s|\right)\, ds\right] \\
\le &\  C\mu_1 E^{\mathscr{F}_t}\left[\int_t^T|\phi'(Y_s^i)|^{2\over 1-\alpha}\, ds\right]
+ C\mu_2 E^{\mathscr{F}_t}\left[\int_t^Te^{{2\gamma \over 1-\alpha}|U_s|}\, ds\right].
  \end{split}
\end{eqnarray}

In view of inequality~\eqref{Young}, we have
\begin{eqnarray}\label{fundamental}
    \begin{split}
    &\sum_{i=1}^n\phi(Y_t^i)+{1\over2} E^{\mathscr{F}_t}\left[\int_t^T|Z_s|^2\, ds\right] \\
             \le\, & C\sum_{i=1}^n\phi(|\xi^i|_\infty)+\left(\beta+C\mu_1\right)\sum_{i=1}^n E^{\mathscr{F}_t}\left[\int_t^T|\phi'(Y_s^i)|^{{2\over 1-\alpha}}\, ds\right]\\
        &+C\mu_2 E^{\mathscr{F}_t}\left[\int_t^Te^{{2\gamma \over 1-\alpha}|U_s|}\, ds\right]+{1\over4}E^{\mathscr{F}_t}\left[\int_t^T|V_s|^2\, ds\right]\\
  \le\, & C\sum_{i=1}^n\phi(|\xi^i|_\infty)+\gamma^{2\over \alpha-1}\left(\beta+C\mu_1\right)\sum_{i=1}^n E^{\mathscr{F}_t}\left[\int_t^T e^{{2\gamma \over 1-\alpha}|Y_s^i|}\, ds\right]\\
        &+C\mu_2 E^{\mathscr{F}_t}\left[\int_t^Te^{{2\gamma \over 1-\alpha}|U_s|}\, ds\right]+{1\over4}E^{\mathscr{F}_t}\left[\int_t^T|V_s|^2\, ds\right].
    \end{split}
\end{eqnarray}

{\it Step 2. Estimate of $e^{\gamma |Y|_\infty}$.}

\medskip
Noting that the solution of the following BSDE
$$
{\bar Y}_t^i=|\xi^i|+\int_t^T(C+{\gamma\over 2}|{\bar Z}_s^i|^2+|h^i(s, U_s, V_s)|)\, ds-\int_t^T{\bar Z}_s^i\, dW_s
$$
is explicitly given by
$$
{\bar Y}_t^i={1\over \gamma}\ln E^{\mathscr{F}_t}\left[e^{\gamma \left(|\xi^i|+\int_t^T(C+|h^i(U_s,V_s)|)\, ds\right)}\right],
$$
we have
$$
e^{\gamma {\bar Y}_t^i }=E^{\mathscr{F}_t}\left[e^{\gamma \left(|\xi^i|+\int_t^T(C+|h^i(U_s,V_s)|)\, ds\right)}\right].
$$

In view of Lemma~\ref{Quad bsde}, comparing the two pairs of solutions $(Y^i, Z^i)$ and $(\bar Y^i, \bar Z^i)$, we have for $i=1,\ldots,n$,
\begin{eqnarray}
    \begin{split}
&e^{{3\over 1-\alpha}\gamma |Y_t^i|}
\le\, e^{{3\over 1-\alpha}\gamma |{\bar Y}_t^i|}\\
\le\, & E^{\mathscr{F}_t}\left[e^{{3\over 1-\alpha}\gamma \left(|\xi^i|+\int_t^T(C+|h^i(U_s,V_s)|)\, ds\right)}\right]\\
\le\, & E^{\mathscr{F}_t}\left[e^{{3\over 1-\alpha}\gamma \left(|\xi|_\infty+C\int_t^T(2+|U_s|+|V_s|^{1+\alpha})\, ds\right)}\right].
\end{split}
\end{eqnarray}
Since (by Young's inequality)
\begin{eqnarray}
    \begin{split}
 {3\over 1-\alpha}\gamma C|V_s|^{1+\alpha} =&  {3\over 1-\alpha}\gamma C\left|\sqrt{{\delta\over n}}V_s\right|^{1+\alpha} \left({\delta\over n}\right)^{{1+\alpha\over 2}}\\
 \le &  {3\over 2} \gamma C \left({n\over \delta}\right)^{1+\alpha\over 2}+{\delta\over n}|V_s|^2,
\end{split}
\end{eqnarray}
in view of the definition of notation $C_\delta$ in~\eqref{C delta}, we have for $i=1,\ldots,n$,
\begin{eqnarray}
    \begin{split}
e^{{3\over 1-\alpha}\gamma |Y_t^i|}\le\, & C_\delta E^{\mathscr{F}_t}\left[e^{\left({3\over 1-\alpha}\gamma|\xi|_\infty+{3\over 1-\alpha}\gamma C\varepsilon |U|_\infty +{\delta \over n}\int_t^T|V_s|^2\, ds\right)}\right];
 \end{split}
\end{eqnarray}
and therefore,
\begin{eqnarray}
    \begin{split}
e^{{3\over 1-\alpha}\gamma |Y_t|}\le\, & C_\delta\ e^{\left({3\over 1-\alpha}n\gamma|\xi|_\infty+{3\over 1-\alpha}n\gamma C\varepsilon |U|_\infty\right)}
E^{\mathscr{F}_t}\left[e^{\delta\int_t^T|V_s|^2\, ds}\right].
\end{split}
\end{eqnarray}
In view of the fact that $\|\sqrt{\delta}V\cdot W\|^2_{BMO_2(P)}\le \delta A<1$ (see~\eqref{root}), applying John-Nirenberg inequality to the BMO martingale $\sqrt{\delta}V\cdot W$, we have
\begin{eqnarray}
    \begin{split}
e^{{3\over 1-\alpha}\gamma |Y_t|}\le\, & {C_\delta \ e^{({3\over 1-\alpha}n\gamma |\xi|_\infty+{3\over 1-\alpha}Cn\gamma \varepsilon |U|_\infty)}\over 1-\delta \|V\cdot W\|^2_{BMO_2}}\le {C_\delta \ e^{({3\over 1-\alpha}n\gamma |\xi|_\infty+{3\over 1-\alpha}Cn\gamma \varepsilon |U|_\infty)}\over 1-\delta A}.
 \end{split}
\end{eqnarray}
Since $3nC\varepsilon\le 1$ (see the choice of $\varepsilon$ in~\eqref{A}) and  $(U,V)\in \mathscr{B}_\varepsilon$, we have
\begin{eqnarray}
    \begin{split}
e^{({3\over 1-\alpha}\gamma |Y|_\infty)}\le\, &  {C_\delta e^{({3\over 1-\alpha}n\gamma |\xi|_\infty+{1\over 1-\alpha} |U|_\infty)}\over 1-\delta A}\\
\le\, & C_\delta { e^{{3\over 1-\alpha}n\gamma |\xi|_\infty}\over 1-\delta A}   \left( C_\delta {e^{{3\over 1-\alpha}n\gamma |\xi|_\infty}\over 1-\delta A}\right)^{1\over2}\\
\le\, & \left({ C_\delta e^{{3\over 1-\alpha}n\gamma |\xi|_\infty}\over 1-\delta A}\right)^{3\over2},
\end{split}
\end{eqnarray}
which gives a half of the desired result~\eqref{desired}.

\medskip
{\it Step 2. Estimate of $\|Z\cdot W\|_{BMO_2}^2$.}

\medskip
From inequality~\eqref{fundamental} and the definition of notation $\mu$, we have
\begin{eqnarray}\label{A*}
    \begin{split}
    {1\over2} E^{\mathscr{F}_t}\left[\int_t^T|Z_s|^2\, ds\right]
  \le\, & {Cn\over \gamma^2}e^{\gamma |\xi|_\infty}+\mu C_\delta{e^{{3\over 1-\alpha}n\gamma |\xi|_\infty}\over 1-\delta A}\varepsilon +{1\over4}A.
\end{split}
\end{eqnarray}
In view of~\eqref{root}, we have
\begin{eqnarray}
    \begin{split}
    \|Z\cdot W\|_{BMO_2}^2\le {Cn\over \gamma^2}e^{\gamma |\xi|_\infty}+\mu C_\delta{e^{{3\over 1-\alpha}n\gamma |\xi|_\infty}\over 1-\delta A}\varepsilon +{1\over4}A= {1\over2} A,
\end{split}
\end{eqnarray}
The other half  of the desired result~\eqref{desired} is then proved.

\subsection{Contraction of the quadratic solution map}
For $(U,V)\in \mathscr{B}_\varepsilon$ and $(\widetilde U, \widetilde V)\in \mathscr{B}_\varepsilon$, set
$$
(Y,Z):=\Gamma(U,V), \quad (\widetilde Y, \widetilde Z):= \Gamma(\widetilde U, \widetilde V). 
$$
That is, for $i=1,\ldots,n,$
\begin{eqnarray}
    \begin{split}
Y_t^i=\, & \xi^i+\int_t^T\left[f^i(Z_s^i)+h^i(U_s,V_s)\right]\, ds -\int_t^TZ_s^i dW_s,\\
{\widetilde Y}_t^i=\, & \xi^i+\int_t^T\left[f^i({\widetilde Z}_s^i)+h^i({\widetilde U}_s,{\widetilde V}_s)\right]\, ds -\int_t^T{\widetilde Z}_s^i dW_s.
\end{split}
\end{eqnarray}
Fix $i$, we can define the vector process $\beta(i)$ in an obvious way such that
\begin{eqnarray}
    \begin{split}
&|\beta_s(i)|\le  C(1+|Z_s^i|+|{\widetilde Z}_s^i|),\\
&{f^i(Z^i)-f^i(\widetilde Z^i)}=  (Z^i-\widetilde Z_s^i)\beta_s(i).
\end{split}
\end{eqnarray}
Then $\widetilde W_t(i):=W_t-\int_0^t\beta_s(i)\, ds$  is a Brownian motion under the equivalent probability measure $P^i$ defined by
$$dP^i: =\mathscr{Exp} (\beta(i)\cdot W)_0^T\, dP,$$
 and
from the above-established a priori estimate that there is $K>0$ such that $\|\beta(i)\cdot W\|^2_{BMO_2}\le K^2:=3C^2T+6C^2A$.

In view of the following equation
\begin{eqnarray}
    \begin{split}
Y_t^i-{\widetilde Y}_t^i +\int_t^T(Z_s^i-{\widetilde Z}_s^i) \, d{\widetilde W}_s(i)=\, & \int_t^T\left[h^i(U_s,V_s)-h^i({\widetilde U}_s,{\widetilde V}_s)\right]\, ds,
\end{split}
\end{eqnarray}
taking square and then the conditional expectation with respect to $P^i$ (denoted by $E_i^{\mathscr{F}_t}$) on both sides of the last equation, we have the following standard estimates:
\begin{eqnarray}\label{square}
    \begin{split}
&|Y_t^i-{\widetilde Y}_t^i|^2+E_i^{\mathscr{F}_t}\left[\int_t^T|Z_s^i-{\widetilde Z}_s^i|^2\, ds\right]\\
= \, & E_i^{\mathscr{F}_t}\left[\left(\int_t^T\left(h^i(U_s,V_s)-h^i({\widetilde U}_s,{\widetilde V}_s)\right)\, ds\right)^2\right]\\
\le\, &C^2E_i^{\mathscr{F}_t}\left[\left(\int_t^T\left(|U_s-{\widetilde U}_s|+(1+|V_s|^\alpha+|{\widetilde V}_s|^\alpha)|V_s-{\widetilde V}_s|\right)\, ds\right)^2\right]\\
\le\, &2C^2(T-t)^2|U-{\widetilde U}|^2_\infty\\
&+2C^2 E_i^{\mathscr{F}_t}\left[\int_t^T (1+|V_s|^{\alpha}+|{\widetilde V}_s|^{\alpha})^2\, ds\int_t^T|V_s-{\widetilde V}_s|^2\, ds\right]\\
\le\, &2C^2(T-t)^2|U-{\widetilde U}|^2_\infty\\
&+6C^2 E_i^{\mathscr{F}_t}\left[\left(\int_t^T (1+|V_s|^{2\alpha}+|{\widetilde V}_s|^{2\alpha})\, ds\right)^2\right]^{1\over2}E_i^{\mathscr{F}_t}\left[\left(\int_t^T|V_s-{\widetilde V}_s|^2\, ds\right)^2\right]^{1\over2}.
\end{split}
\end{eqnarray}
Let $L_4$ be a generic constant for the following dominance:
$$
\sup_{\tau}E^{\mathscr{F}_\tau}\left[(\langle M\rangle)^2\right] \le L_4^2 E \|M\|_{BMO}^2:=L_4^2  \sup_{\tau}E^{\mathscr{F}_\tau}\left[\langle M\rangle\right]
$$
for any BMO martingale $M$. We have
\begin{eqnarray}
    \begin{split}
E_i^{\mathscr{F}_t}\left[\left(\int_t^T|V_s-{\widetilde V}_s|^2\, ds\right)^2\right]^{1\over2}\le & \ L_4^2 \|(V-{\widetilde V})\cdot W\|^2_{BMO_2(P^i)}\\
\le &\  L_4^2c_2^2\|(V-{\widetilde V})\cdot W\|^2_{BMO_2(P)}
\end{split}
\end{eqnarray}
and for $t\in [T-\varepsilon, T]$,
\begin{eqnarray}
    \begin{split}
&\ E_i^{\mathscr{F}_t}\left[\left(\int_t^T (1+|V_s|^{2\alpha}+|{\widetilde V}_s|^{2\alpha})\, ds\right)^2\right]^{1\over2}\\
\le &\ E_i^{\mathscr{F}_t}\left[\left(\varepsilon+\varepsilon^{1-\alpha}\left(\int_t^T |V_s|^2\, ds\right)^{\alpha}+\varepsilon^{1-\alpha}\left(\int_t^T|{\widetilde V}_s|^2\, ds\right)^{\alpha}\right)^2\right]^{1\over2}\\
\le &\ \varepsilon^{1-\alpha} E_i^{\mathscr{F}_t}\left[\left(\varepsilon^\alpha+\left(\int_t^T |V_s|^2\, ds\right)^{\alpha}+\left(\int_t^T|{\widetilde V}_s|^2\, ds\right)^{\alpha}\right)^2\right]^{1\over2}\\
\le &\ \varepsilon^{1-\alpha} E_i^{\mathscr{F}_t}\left[\left(T^\alpha+2-2\alpha+\alpha\int_t^T |V_s|^2\, ds+\alpha\int_t^T|{\widetilde V}_s|^2\, ds\right)^2\right]^{1\over2}\\
\le &\ \varepsilon^{1-\alpha} \left[T^\alpha+2+\alpha E_i^{\mathscr{F}_t}\left[\left(\int_t^T |V_s|^2\, ds\right)^2\right]^{1\over2}+\alpha E_i^{\mathscr{F}_t}\left[\left(\int_t^T|{\widetilde V}_s|^2\, ds\right)^2\right]^{1\over2}\right]\\
\le &\ \varepsilon^{1-\alpha} \left[T^\alpha+2+\alpha L_4^2\|V\cdot W\|^2_{BMO_2(P^i)}+\alpha L_4^2\|{\widetilde V}\cdot W\|^2_{BMO_2(P^i)}\right]\\
\le &\ \varepsilon^{1-\alpha} \left[T^\alpha+2+\alpha L_4^2c_2^2\|V\cdot W\|^2_{BMO_2(P)}+\alpha L_4^2c_2^2\|{\widetilde V}\cdot W\|^2_{BMO_2(P)}\right]\\
\le &\ \varepsilon^{1-\alpha} \left(T^\alpha+2+2\alpha L_4^2c_2^2A\right).
\end{split}
\end{eqnarray}
Concluding the above estimates, we have for $t\in [T-\varepsilon, T]$,
\begin{eqnarray}
    \begin{split}
  &|Y_t^i-{\widetilde Y}_t^i|^2+E_i^{\mathscr{F}_t}\left[\int_t^T|Z_s^i-{\widetilde Z}_s^i|^2\, ds\right]\\
\le\, &2C^2\varepsilon^2|U-{\widetilde U}|^2_\infty\\
&+6C^2 L_4^2c_2^2\left(T^\alpha+2+2\alpha L_4^2c_2^2A\right)\varepsilon^{1-\alpha} \|(V-{\widetilde V})\cdot W\|^2_{BMO_2(P)}.
\end{split}
\end{eqnarray}
In view of estimates~\eqref{uniform BMO estimate}, we have for $t\in [T-\varepsilon, T]$,
\begin{eqnarray}
    \begin{split}
&|Y-{\widetilde Y}|^2_\infty+c_1^2\|(Z-{\widetilde Z})\cdot W\|^2_{BMO_2(P)}\\
\le \, &4C^2n\ \varepsilon^2|U-{\widetilde U}|^2_\infty
+12C^2 L_4^2c_2^2n\left(T^\alpha+2+2\alpha L_4^2c_2^2A\right)\varepsilon^{1-\alpha} \|(V-{\widetilde V})\cdot W\|^2_{BMO_2(P)}.
\end{split}
\end{eqnarray}

Now it is standard to derive the desired results.

\section{Global solution: the proof of Theorem~\ref{thm 2}}

From Cheriditi and Nam~\cite[Lemma 4.3, page 13]{CheriditoNam}, any local solution $(Y', Z')$ on $[T-\eta, T]$ ($\eta>0$) of BSDE has the following uniform estimate:
\begin{equation}\label{uniform estimate}
|Y'_t|\le (C+1)e^{{1\over2}(C+1)^2(T-t)}, \quad t\in [T-\eta, T]
\end{equation}
with $C$ depending on the bound of the absolute terminal value $|\xi|$ (being independent of $\eta$). This assertion together with our theorem on the local existence will be used to show in what follows that BSDE has a global solution $(Y, Z)$ on $[0,T]$.

Define
$$
\lambda :=(C+1)e^{{1\over2}(C+1)^2T}.
$$
Our theorem shows that there is $\eta_\lambda>0$ which only depends on $\lambda$, such that BSDE has a local solution $(Y, Z)$ on $[T-\eta_\lambda, T]$. Then, we have from the above estimate~\eqref{uniform estimate}
$$
|Y_{T-\eta_\lambda}|\le \lambda.
$$
Taking $T-\eta_\lambda$ as the terminal time, our theorem shows that BSDE has a local solution $(Y, Z)$ on $[T-2\eta_\lambda, T-\eta_\lambda]$. Then, viewing $(Y, Z)$ as a local solution on $[T-2\eta_\lambda, T]$, we have from the above estimate~\eqref{uniform estimate}
$$
|Y_{T-2\eta_\lambda}|\le \lambda.
$$
Repeating the preceding process, we can extend the pair $(Y, Z)$ to the whole interval $[0, T]$ within a finite steps such that $Y$ is uniformly bounded by $\lambda$. We now show that $Z\cdot W$ is a $BMO(P)$ martingale.

Identical to the proof of inequality~\eqref{quadratic estimate}, we have for $i=1,\ldots,n,$
 \begin{eqnarray}
    \begin{split}
      &\phi(Y_t^i)+{1\over2} E^{\mathscr{F}_t}\left[\int_t^T|Z_s^i|^2\, ds\right] \\
      \le\, & \phi(|\xi^i|_\infty)+ CE^{\mathscr{F}_t}\left[\int_t^T|\phi'(Y_s^i)|\left(2+|Y_s|+|Z_s|\right)\, ds\right]\\
       \le\, & \phi(|\xi|_\infty)+ C\phi'(\lambda)E^{\mathscr{F}_t}\left[\int_t^T\left(2+\lambda+|Z_s|\right)\, ds\right].
    \end{split}
\end{eqnarray}
Consequently, we have
\begin{eqnarray}
    \begin{split}
      & \|Z\cdot W\|^2_{BMO_2(P)}=\sup_{\tau}E^{\mathscr{F}_\tau}\left[\int_\tau^T|Z_s|^2\, ds\right] \\
            \le\, & 2n\phi(|\xi|_\infty)+ 2Cn\phi'(\lambda)(2+\lambda)T+ 2Cn\phi'(\lambda)\sqrt{T}\|Z\cdot W\|_{BMO_2(P)},
    \end{split}
\end{eqnarray}
which implies the following bound of  $\|Z\cdot W\|_{BMO_2(P)}$:
\begin{eqnarray}
    \begin{split}
      \|Z\cdot W\|_{BMO_2(P)}\le Cn\phi'(\lambda)\sqrt{T}+ \sqrt{2n\phi(|\xi|_\infty)+ 2Cn\phi'(\lambda)(2+\lambda)T}.
    \end{split}
\end{eqnarray}

Finally, we prove the uniqueness. Let $(Y,Z)$ and $(\widetilde Y, \widetilde Z)$ be two adapted solutions. Then, we have
\begin{eqnarray}
    \begin{split}
Y_t^i-{\widetilde Y}_t^i=\, & \int_t^T\left(h^i(Y_s,Z_s)-h^i({\widetilde Y}_s,{\widetilde Z}_s)\right)\, ds -\int_t^T(Z_s^i-{\widetilde Z}_s^i) \, d{\widetilde W}_s(i).
\end{split}
\end{eqnarray}
%Using It\^o's formula, we have the following estimates:
%\begin{eqnarray}
%    \begin{split}
%&|Y_t^i-{\widetilde Y}_t^i|^2+E_i^{\mathscr{F}_t}\left[\int_t^T|Z_s^i-{\widetilde Z}_s^i|^2\, ds\right]\\
%= \, & 2 E_i^{\mathscr{F}_t}\left[\int_t^T\left(Y_s^i-{\widetilde Y}_s^i\right)\left(h^i(Y_s,Z_s)-h^i({\widetilde Y}_s,{\widetilde Z}_s)\right)\, ds\right]\\
%\le\, &2CE_i^{\mathscr{F}_t}\left[\int_t^T\left|Y_s^i-{\widetilde Y}_s^i\right|\left(|Y_s-{\widetilde Y}_s|+|Z_s-{\widetilde Z}_s|\right)\, ds\right]\\
%\le\, &E_i^{\mathscr{F}_t}\left[\int_t^T\left((2C+C^2\epsilon^{-1})|Y_s-{\widetilde Y}_s|^2+\epsilon |Z_s-{\widetilde Z}_s|^2\right)\, ds\right]\\
%\le\, &(2C+C^2\epsilon^{-1})\int_t^T|Y_s-{\widetilde Y}_s|^2_\infty\, ds+\epsilon \|(Z-{\widetilde Z})\cdot W\|^2_{BMO_2(P^i)}\\
%\le\, &(2C+C^2\epsilon^{-1})\int_t^T|Y_s-{\widetilde Y}_s|^2_\infty\, ds+c_2\epsilon \|(Z-{\widetilde Z})\cdot W\|^2_{BMO_2(P)}.
% \end{split}
%\end{eqnarray}
%Using Gronwall's inequality, it is standard to show the following:
%\begin{eqnarray}
%    \begin{split}
%|Y_t^i-{\widetilde Y}_t^i|^2_\infty
%\le\, & c_2\epsilon e^{(2C+C^2\epsilon^{-1})(T-t)}\|(Z-{\widetilde Z})\cdot W\|^2_{BMO_2(P)}.
% \end{split}
%\end{eqnarray}
In our case of $\alpha=0$, similar to the first two inequalities in \eqref{square}, for any stopping time $\tau$ which takes values in $[T-\varepsilon, T]$, we have
\begin{eqnarray}
    \begin{split}
&|Y_\tau^i-{\widetilde Y}_\tau^i|^2+E_i^{\mathscr{F}_\tau}\left[\int_\tau^T|Z_s^i-{\widetilde Z}_s^i|^2\, ds\right]\\
= \, & E_i^{\mathscr{F}_\tau}\left[\left(\int_\tau^T\left(h^i(Y_s,Z_s)-h^i({\widetilde Y}_s,{\widetilde Z}_s)\right)\, ds\right)^2\right]\\
\le\, &C^2E_i^{\mathscr{F}_\tau}\left[\left(\int_\tau^T\left(|Y_s-{\widetilde Y}_s|+|Z_s-{\widetilde Z}_s|\right)\, ds\right)^2\right]\\
\le \, &2C^2\varepsilon^2|Y-\widetilde Y|_\infty^2+2C^2 \varepsilon E_i^{\mathscr{F}_\tau}\left[\int_\tau^T|Z_s-{\widetilde Z}_s|^2\, ds\right]\\
\le \, &2C^2\varepsilon^2|Y-\widetilde Y|_\infty^2+2C^2 \varepsilon \|(Z-{\widetilde Z})\cdot \widetilde W(i)\|^2_{BMO_2(P^i)}\\
\le \, &2C^2\varepsilon^2|Y-\widetilde Y|_\infty^2+2C^2 c_2^2\varepsilon \|(Z-{\widetilde Z})\cdot  W\|^2_{BMO_2(P)}.
\end{split}
\end{eqnarray}
Therefore, we have (on the interval $[T-\varepsilon, T]$)
\begin{eqnarray}
    \begin{split}
&\left|Y-{\widetilde Y}\right|^2_\infty+c_1^2\left\|(Z-{\widetilde Z})\cdot W\right\|^2_{BMO_2(P)} \\
\le \, &4C^2n\varepsilon^2\left|Y-\widetilde Y\right|_\infty^2+4C^2 c_2^2n\varepsilon \left\|(Z-{\widetilde Z})\cdot  W\right\|^2_{BMO_2(P)}.
\end{split}
\end{eqnarray}
Note that since
$$\beta(i)\le C(1+|Z^i|+|\widetilde Z^i|), \quad i=1,\ldots,n,$$
the two generic constants $c_1$ and $c_2$ only depend on the sum
$$
\left\|Z\cdot  W\right\|^2_{BMO_2(P)}+\left\|{\widetilde Z}\cdot  W\right\|_{BMO_2(P)}.
$$
Then when $\varepsilon$ is sufficiently small, we conclude that  $Y=\widetilde Y$ and $Z=\widetilde Z$ on $[T-\varepsilon, T]$. Repeating iteratively with a finite of times, we have the uniqueness
on the given interval $[0, T]$.

\bibliographystyle{siam}

\end{document}